\documentclass[]{article}

\usepackage[pdftex]{graphicx}

\usepackage{amsmath}
\usepackage{amssymb}
\usepackage{amsthm}

\newcommand{\fn}{\ensuremath{\mathcal{F}_{n}}}

\newcommand{\suff}{\ensuremath{\mathcal{K}_{m}}}

\newcommand{\no}[1]{\ensuremath{\epsilon(#1)}}
\newcommand{\yes}[1]{\ensuremath{\delta(#1)}}
\newcommand{\gs}[3]{\ensuremath{#1^{#2}_{#3}}}
\newcommand{\denom}[1]{\ensuremath{d_{#1}(x)}}
\newcommand{\cd}{\ensuremath{\cdot}}

\newtheorem{theorem}{Theorem}
\newtheorem{corollary}[theorem]{Corollary}
\newtheorem{lemma}[theorem]{Lemma}

\DeclareMathOperator{\sk}{sk}

\begin{document}

\title{
Generalized Fibonacci recurrences and the 
lex-least De Bruijn sequence}
\author{Joshua Cooper \and Christine E. Heitsch}
\date{\today}
\maketitle

\begin{abstract}
The skew of a binary string is the difference between 
the number of zeroes and the number of ones, while the length of the string
is the sum of these two numbers.
We consider certain suffixes of the lexicographically-least de 
Bruijn sequence at natural breakpoints of the binary string.
We show that the skew and length of these suffixes are enumerated by 
sequences generalizing the Fibonacci and Lucas numbers, respectively.
\end{abstract}

\section{Introduction}\label{intro}

Let $w = a_{1} a_{2} \ldots a_{l}$ be a word over the alphabet $\{0,1\}$
of length $|w| = l$.
When $|w| = 2^{n}$ and 
the indices of $w$ are interpreted cyclically, the word is said to be a 
\textit{binary de Bruijn sequence of order $n$} if it contains 
each of the $2^{n}$ distinct binary strings of length $n$ as a subword. 
The string $00010111$ is a binary de Bruijn sequences of order $3$.

A \textit{binary necklace} is an equivalence class of binary words 
under rotation.
The representative element for the equivalence class is chosen to be 
the lexicographically least one.
A binary string is a \textit{Lyndon word} if it is an aperiodic necklace
representative.
The binary Lyndon words of length $\leq 4$ are
$0$, $1$, $01$, $001$, $011$, $0001$, $0011$, $0111$.  

De Bruijn sequences and Lyndon words are related via the ``Ford sequence,''
denoted here \fn, which is the lexicographically least binary de Bruijn 
sequence of order $n$.
Fredricksen proved~\cite{fredricksen-82} that \fn\ is obtained 
by concatenating all Lyndon words of lengths dividing $n$ in
lexicographic order. 
For instance, $\mathcal{F}_{4} = 0000100110101111$.
We note that this result generalizes to constructing the 
lexicographically-least de Bruijn sequence over an arbitrary 
alphabet~\cite{fredricksen-maiorana-78, moreno-04}.

The Ford sequence is also generated by applying a greedy strategy to 
the production of a binary de Bruijn sequence. 
The algorithm constructs \fn\ one bit at a time, preferring $0$'s to 
$1$'s whenever possible.
Given this, it is reasonable to expect that initial segments of \fn\ 
contain many more zeros than ones. 
In fact, previous work~\cite{ford1} shows that the maximum difference (called
the discrepancy) of \fn\ is $\Theta(2^n \log n/n)$.

The discrepancy is the maximum possible ``skew'' over all prefixes  of \fn.
The {\it skew} of a binary string $w$ of length $l$, 
denoted $\sk(w)$, is the difference 
between the number of zeros and the number of ones.
Since the length of $w$ is the sum of these two numbers, we have that
\[ \sk(w) = \sum_{i = 1}^{l} (-1)^{a_{i}} 
\mbox{ and }
|w| = \sum_{i=1}^{l} (1)^{a_{i}}. \]

Figure~\ref{fslabels} illustrates the discrepancy for $n = 4,5,6,7$ 
by graphing the skew of all prefixes of \fn.
As illustrated on the graphs,
there are natural breakpoints in \fn\
following the occurrence of the subword $0^i 1^{n-i}$ for $1 \leq i \leq n -1$.
The cases where $i = 0$ and $i = n$ are the final $1$ and initial $0$,
respectively.

\begin{figure}
\centering
\includegraphics[width = .45\textwidth]{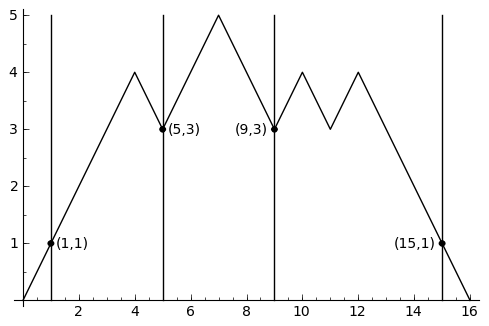}
\includegraphics[width = .45\textwidth]{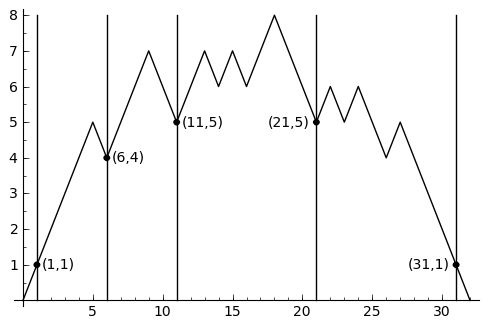} \\
\includegraphics[width = .45\textwidth]{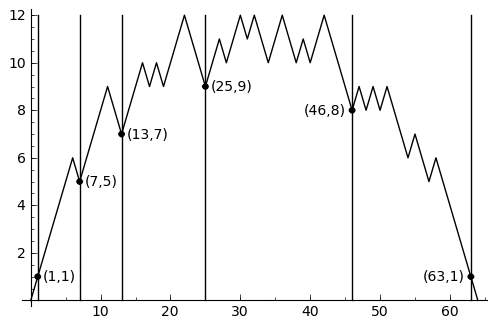}
\includegraphics[width = .45\textwidth]{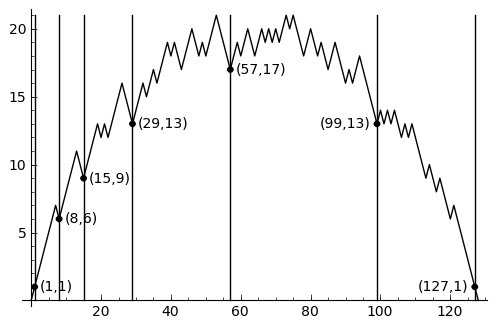}
\caption{\label{fslabels} The discrepancy of the Ford sequence \fn\ for 
$n = 4,5,6,7$ with particular ``breakpoint'' values of the prefix skew 
identified.}
\end{figure}

In this article, 
we prove that the skew at these breakpoints gives sequences of values 
which are Fibonacci-like.
Our results are given in terms of suffixes of the Ford sequence, 
which are directly related to prefixes by $\sk(\fn) = 0$ and $|\fn| = 2^n$.
As stated precisely in Theorem~\ref{main} below, 
we show that the skew and the length of these breakpoint suffixes of \fn\ 
are enumerated by sequences generalizing the 
Fibonacci and Lucas numbers, respectively.

\section{Preliminaries and statement of main result}\label{prelim}

Let $\ell_{0} = 1$ and,
for $1 \leq i \leq n - 1$, let $\ell_{i}$ be the subword of \fn\ 
which begins immediately after the Lyndon word
$0^{i+1}1^{n-i-1}$ and ends with the string $0^{i} 1^{n-i}$.
Hence, $\ell_{i}$ consists of 
all Lyndon words of length $d \mid n$, 
in lexicographic order,
which contain the substring $0^i$ but not $0^{i+1}$.
For technical reasons,
if $i > n - 1$, then we define $\ell_{i} = \varepsilon$, the empty string.

Let $\mathcal{L}_{m}$ be the concatenation of 
$\ell_{m} \ell_{m-1} \ldots \ell_{1}$ for a given \fn.
Hence, $\mathcal{L}_{m}$ is the substring of the $n$-th Ford sequence which
contains the Lyndon words of length $d > 1$ with at least one $0$ and 
at most $m$ consecutive $0$'s. 
Let \suff\ be the proper suffix of \fn\ consisting of the 
Lyndon words of length $d \mid n$ containing at most $m$ consecutive $0$'s;
\[ \suff = \mathcal{L}_{m} 1 = \ell_{m} \ell_{m-1} \ldots \ell_{1} \ell_{0} \mbox{.} \]
If $m \geq n - 1$, then $\suff$ is $\mathcal{F}_{n}$ except for the initial
$0$.
Also, $\mathcal{K}_0 = \ell_0 = 1$.

\begin{table}
\centering
\begin{tabular}{c|*{10}{c}}
& \multicolumn{10}{c}{$\sk(\suff)$} \\
\fn & 0 & 1 & 2 & 3 & 4 & 5 & 6 & 7 & 8 & 9 \\ \hline
1 & -1 & \cd & \cd & \cd & \cd & \cd & \cd & \cd & \cd & \cd \\ 
2 & -1 & -1 & \cd & \cd & \cd & \cd & \cd & \cd & \cd & \cd \\ 
3 & -1 & -2 & -1 & \cd & \cd & \cd & \cd & \cd & \cd & \cd \\ 
4 & -1 & -3 & -3 & -1 & \cd & \cd & \cd & \cd & \cd & \cd \\ 
5 & -1 & -5 & -5 & -4 & -1 & \cd & \cd & \cd & \cd & \cd \\ 
6 & -1 & -8 & -9 & -7 & -5 & -1 & \cd & \cd & \cd & \cd \\ 
7 & -1 & -13 & -17 & -13 & -9 & -6 & -1 & \cd & \cd & \cd \\ 
8 & -1 & -21 & -31 & -25 & -17 & -11 & -7 & -1 & \cd & \cd \\ 
9 & -1 & -34 & -57 & -49 & -33 & -21 & -13 & -8 & -1 & \cd \\ 
10 & -1 & -55 & -105 & -94 & -65 & -41 & -25 & -15 & -9 & -1 \\
\end{tabular}
\caption{Values of $\sk(\suff)$ when $0 \leq m \leq 9$ for each \fn\ with 
$1 \leq n \leq 10$.} 
\end{table}

\begin{table}
\centering
\begin{tabular}{c|*{10}{c}}
& \multicolumn{10}{c}{$|\suff|$} \\
\fn & 0 & 1 & 2 & 3 & 4 & 5 & 6 & 7 & 8 & 9 \\ \hline
1 & 1 & \cd & \cd & \cd & \cd & \cd & \cd & \cd & \cd & \cd \\ 
2 & 1 & 3 & \cd & \cd & \cd & \cd & \cd & \cd & \cd & \cd \\ 
3 & 1 & 4 & 7 & \cd & \cd & \cd & \cd & \cd & \cd & \cd \\ 
4 & 1 & 7 & 11 & 15 & \cd & \cd & \cd & \cd & \cd & \cd \\ 
5 & 1 & 11 & 21 & 26 & 31 & \cd & \cd & \cd & \cd & \cd \\ 
6 & 1 & 18 & 39 & 51 & 57 & 63 & \cd & \cd & \cd & \cd \\ 
7 & 1 & 29 & 71 & 99 & 113 & 120 & 127 & \cd & \cd & \cd \\ 
8 & 1 & 47 & 131 & 191 & 223 & 239 & 247 & 255 & \cd & \cd \\ 
9 & 1 & 76 & 241 & 367 & 439 & 475 & 493 & 502 & 511 & \cd \\ 
10 & 1 & 123 & 443 & 708 & 863 & 943 & 983 & 1003 & 1013 & 1023 \\
\end{tabular}
\caption{Values of $|\suff|$ when $0 \leq m \leq 9$ for each \fn\ with 
$1 \leq n \leq 10$.} 
\end{table}

The \textit{Fibonacci numbers} are defined by the recurrence 
$F_{n} = F_{n-1} + F_{n-2}$ with initial conditions $F_{0} = 0$ and $F_{1} = 1$.
The \textit{Lucas numbers} are defined by the recurrence 
$L_{n} = L_{n-1} + L_{n-2}$ with initial conditions $L_{0} = 2$ and $L_{1} = 1$.
The (ordinary) generating functions for these sequences are 
$x / (1 - x - x^{2})$ and $(2-x) / (1 - x - x^{2})$, respectively.
For a detailed treatment of generating functions for recurrence relations,
we refer the reader to~\cite{wilf-94}. 

Let 
\[ \denom{m} = 1 - x - x^{2} -\ldots - x^{m} = 1 - x \sum_{i = 0}^{m-1} x^{i}.\]

Let $\gs{G}{m}{n}$ be the integer sequence defined by the $m$-th order recurrence
$\gs{G}{m}{n} = \sum_{i = 1}^{m} \gs{G}{m}{n-i}$ with initial conditions
$\gs{G}{m}{0} = \gs{G}{m}{1} = \ldots = \gs{G}{m}{m-1} = 1$.
This is a generalization of the Fibonacci numbers, 
and when $m = 2$ we recover $\gs{G}{2}{n} = F_{n+1}$.
It is straightforward to see that
the sequence $\gs{G}{m}{n}$ has the generating function
\[ \frac{1 - \sum_{i = 2}^{m-1} (i-1) x^{i}}{\denom{m}}.\]
There are many possible generalizations of the Fibonacci numbers, 
depending on how the initial conditions $F_{0} = 0$ and $F_{1} = 1$ 
(and, in this case, $F_{2} = 1$) are extended.  

There are likewise different generalizations of the Lucas numbers.
Let $\gs{H}{m}{n}$ be the sequence defined by the $m$-th order recurrence
$\gs{H}{m}{n} = \sum_{i = 1}^{m} \gs{H}{m}{n-i}$ with initial conditions
$\gs{H}{m}{0} = m$ and $\gs{H}{m}{i} = 2^{i} - 1$ for $1 \leq i \leq m - 1$.
So $\gs{H}{2}{n} = L_{n}$, and it is straightforward to see that
the generating function for the sequence $\gs{H}{m}{n}$ is 
\[\frac{m - \sum_{i = 1}^{m-1}(m - i) x^{i}}{\denom{m}}.\] 

In this article, we prove the following.
\begin{theorem}\label{main}
For \fn\ with $n > 0$ and $m \geq 0$, 
\[ \sk(\suff) = -\gs{G}{m + 1}{n - 1} 
\mbox{ and } |\suff| = \gs{H}{m+1}{n} .\]
\end{theorem}

\section{Fibonacci, Lucas, De Bruijn, and Lyndon}

We first prove Theorem~\ref{main} for the special case when $m = 1$.
That the result holds when $m = 0$ follows directly from the definitions.

We define a Lyndon word $w$ to be a \textit{primitive of order $i$} 
if $w = 0^{i} 1^{j}$ with $i + j = |w|$, $i, j \geq 1$.
If a Lyndon word is not primitive, we say it is \textit{composite}.

Let $w$ be a Lyndon word of length $d \mid n$ which occurs in $\ell_{1}$.
Then $w$ may be uniquely parsed into primitives of order $1$ as
\[ w = 0 1^{j_{1}} 0 1^{j_{2}} \ldots 0 1^{j_{k}}
\mbox{, where $j_{l} \geq 1$ and $\sum_{l = 1}^{k} (1 + j_{l}) = d$.} \]
Let $\phi$ be a mapping from primitives of order 1 into 
the integers where $\phi(0 1^j) = 1 + j$ 
and let $\Phi(n)$ be the multiset obtained by applying $\phi$ to the 
$0 1^{j}$ subwords of $\ell_{1}$ from \fn.
For instance, in $\mathcal{F}_{6}$ we have $\ell_{1} = 01010111011011111$
and $\Phi(6) = \{2, 2, 3, 4, 6\}$.

Let $c(n,k)$ be the number of integers $k \geq 2$ in the multiset $\Phi(n)$.
Since each $\ell_{1}$ from \fn\ with $n > 1$ 
contains exactly one primitive of order 1 and length $n$,
we have that $c(n, k) = 1$ when $n = k$.
Also, $c(n, k) = 0$ for $n < k$. 
As we show below,
the distribution for other $(n,k)$ is Fibonacci-like.

Recall that a \textit{composition} of an integer $n$ into 
$k$ (positive) parts is an ordered sum of integers  
\[ n = x_{1} + x_{2} + \ldots + x_{k} \mbox{ where $x_{i} \geq 1$.} \]
We denote such a composition of $n$ as an ordered $k$-tuple 
$x = (x_{1}, x_{2}, \ldots, x_{k})$.

Since primitives of order $1$ have length $k \geq 2$, in the proof below
we consider compositions
of $n - k$ with parts greater than $1$.
There are $F_{n - k - 1}$ distinct 
$x = (x_{1}, \ldots, x_{j})$, with $\sum_{i = 1}^{j} x_{i} = n - k$ and  
$x_{i} \geq 2$, a fact easily obtained by induction on $n-k$.
We show that there are an equal number of distinct $01^{k-1}$ primitives 
in the substring $\ell_{1}$ of \fn. 

\begin{theorem}\label{fibcounts}
\[ c(n, k) = F_{n-k-1} \mbox{ for } 2 \leq k \leq n -1\mbox{.}\]
\end{theorem}

\begin{proof}
Let $k \geq 2$ be fixed.
For $n = k + 1$, there cannot be a primitive of order $1$ and length $k$ 
in $\ell_1$ of $\mathcal{F}_{k+1}$, so $c(k+1, k) = 0$.

Let $n \geq k + 2$.
Consider the compositions of $n - k$ with parts greater than 1.
For each such $x$, let $\omega(x)$ be the binary string 
\[ 0 1^{k-1} 0 1^{x_{1} - 1} \ldots 0 1^{x_{j} - 1}. \]

Suppose that $\omega(x)$ is aperiodic, and 
let $\lambda(x)$ be the Lyndon word which is the representative 
element for the equivalence class under rotation of $\omega(x)$.
If $k$ does not occur in $x$,
then $\lambda(x)$ 
contributes exactly one integer $k$ to the multiset $\Phi(n)$.

Otherwise, $\lambda(x)$ contributes $m + 1$ times to the count of $c(n,k)$,
where there are $m \geq 1$ parts of $x$ which equal $k$.
In this case, there are $m$ additional compositions 
$x^{j_{1}}, \ldots, x^{j_{m}}$
of $n - k$ 
such that 
$\omega(x^{j_{1}}), \ldots, \omega(x^{j_{m}})$ all belong to the 
equivalence class of $\omega(x)$ under rotation.
Hence,
there are $m + 1$ compositions of $n - k$ which are associated with
the same Lyndon word of length $n$ from $\ell_{1}$.

Suppose now that $\omega(x)$ is periodic with period $p$.
Let $\lambda(x)$ be the Lyndon word of length $p$ such that 
$(\lambda(x))^{n/p}$ is an element of the 
rotational equivalence class of $\omega(x)$.
Then there are $q(n/p) - 1$ parts of $x$ which equal $k$ for 
some $q \geq 1$.
Hence, $\lambda(x)$ contributes $q$ integers $k$ to the 
multiset $\Phi(n)$.
Observe that if $q = 1$, then $x$ is the only composition of $n - k$
associated with $\lambda(x)$.
Otherwise, there are $q - 1$ other instances of $01^{k-1}$ in $\lambda(x)$
and $q - 1$ distinct compositions of $n - k$ which are associated 
with $\lambda(x)$.
\end{proof}

Recall that the skew of a binary string $w$ is the difference between 
the number of $0$'s in $w$ and the number of $1$'s, denoted here
$\no{w}$ and $\yes{w}$ respectively.
Hence, 
\[ \sk(w) = \no{w} - \yes{w}.\]

\begin{theorem}
For \fn\ with $n \geq 2$, 
$\no{\ell_{1}} = F_{n-1}$ and $\yes{\ell_{1}} = F_{n+1} - 1$.
\end{theorem}

\begin{proof}
The result follows from Theorem~\ref{fibcounts}, and the identities
\[ \sum_{i = 1}^{n} F_{i} = F_{n+2} - 1 \]
and 
\[ \sum_{i = 1}^{n} i \cdot F_{n - i} = F_{n+3} - (n + 2).\]
Each primitive of order $1$ and length $k$ contributes a 
zero to $\no{\ell_{1}}$ and $k - 1$ ones to $\yes{\ell_{1}}$.
Hence, 
\begin{eqnarray*}
\no{\ell_{1}} & = & \sum_{k = 2}^{n} c(n, k) \\
	& = & 1 + \sum_{k = 2}^{n-1} F_{n-k-1} \\
	& = & F_{n-1} \mbox{.} 
\end{eqnarray*}
Likewise, we have that 
\begin{eqnarray*}
\yes{\ell_{1}} & = & \sum_{k = 2}^{n} (k - 1) \cdot c(n, k) \\
	& = & (n-1) - \sum_{k = 2}^{n-1} F_{n-k-1} + \sum_{k = 2}^{n-1} k \cdot F_{n-k-1} \\
	& = & (n-1) - (F_{n-1} - 1) - F_{n - 2} + \sum_{k = 1}^{n-1} k \cdot F_{n - 1 - k} \\ 
	& = & n - F_{n} + F_{n + 2} - (n + 1) \\
	& = & F_{n + 1} - 1. 
\end{eqnarray*}
\end{proof}

Because the Lucas and Fibonacci numbers are related as 
$L_{n} = F_{n-1} + F_{n+1}$ for $n \geq 1$,
we have the following result.
\begin{corollary}
For \fn\ with $n \geq 2$, 
$\sk(\ell_{1}) = - F_{n} + 1$ and $|\ell_{1}| = L_{n} - 1$.
\end{corollary}
Since $\mathcal{K}_{1} = \ell_{1}1$ for \fn\ with $n \geq 1$, 
we know that $\sk(\mathcal{K}_1) = -\gs{G}{2}{n-1}$ and 
$|\mathcal{K}_1| = \gs{H}{2}{n}$.
Hence, Theorem~\ref{main} holds for $m = 0, 1$.

\section{Generalizing to higher orders}

We generalize compositions of an integer $n$ into parts greater than $1$
to accommodate Lyndon word primitives $0^{i} 1^{j}$ of order $i \geq 1$.
We use the notation $x^{(y)} = (x^{(y-1)})'$ where $x^{(0)} = x$, so 
$x^{(1)} = x'$, $x^{(2)} = x''$, etc.
We say that 
\[ n = x_{1}^{(y_{1})} + x_{2}^{(y_{k})} + \ldots + x_{k}^{(y_{k})} \]
is an \textit{$m$-colored composition of $n$ into $k$ parts greater than $1$}
if 
\begin{itemize}
\item $\sum_{i = 1}^{k} x_{i} = n$ with $x_{i} \geq 2$,
\item and $0 \leq y_{i} \leq \min\{x_{i} - 2, m - 1\}$ for $1 \leq i \leq k$.
\end{itemize}
Note that $x_{i}^{(y_{i})} = x_{j}^{(y_{j})}$ if and only if 
$x_{i} = x_{j}$ and $y_{i} = y_{j}$.
For instance,
$5$, $5'$, $2 + 3$, $3 + 2$, $3' + 2$, and $2 + 3'$ are the six 
$2$-colored compositions of $5$.

For $m \geq 2$,
let $\gs{P}{m}{n}$ be the sequence defined by the $m$-th order recurrence 
$\gs{P}{m}{n} = \sum_{i = 1}^{m} \gs{P}{m}{n - i}$ with initial conditions
$\gs{P}{m}{0} = \ldots = \gs{P}{m}{m-3} = 0$, $\gs{P}{m}{m-2} = 1$, and 
$\gs{P}{m}{m-1} = 0$.

The sequence $\gs{P}{m}{n}$ is another generalization of the Fibonacci numbers.
For $m = 2$, we have that $\gs{P}{2}{n} = F_{n-1}$ when $n \geq 1$.
In general, an inductive argument from the definition yields the following 
identities.

\begin{lemma}
Let $m \geq 2$.
For $m \leq n \leq 2m - 2$, $\gs{P}{m}{n} = 2^{n-m}$.
Also, $\gs{P}{m}{2m-1} = 2^{m-1} - 1$.
\end{lemma}

Let $d(m, n)$ be 
the number of $m$-colored compositions of $n$ into parts greater than $1$
for $m, n \geq 1$.

\begin{theorem}
\[ d(m, n) = \gs{P}{m + 1}{n + m - 1}.\]
\end{theorem}

\begin{proof}
The $m$-colored compositions of $n$ satisfy an $(m+1)$-th order recursion
as follows.
Let $$x = (x_{1}^{(y_{1})}, x_{2}^{(y_{2})}, \ldots, x_{k}^{(y_{k})})$$ for 
$\sum_{i = 1}^{k} x_{i} = n$ with integers $x_{i} \geq 2$ 
and $0 \leq y_{i} \leq \min\{x_{k} - 2, m - 1\}$.

Suppose $n \geq m + 2$.
If $y_{k} < x_{k} - 2$, then 
$$(x_{1}^{(y_{1})}, x_{2}^{(y_{2})}, \ldots, (x_{k} - 1)^{(y_{k})})$$
is an $m$-colored composition of $n - 1$.
Otherwise, $2 \leq x_{k} \leq m + 1$ and 
$$(x_{1}^{(y_{1})}, x_{2}^{(y_{2})}, \ldots, x_{k-1}^{(y_{k-1})})$$
is an $m$-colored composition of $n - x_{k}$.
For $n = m + 2$, the recurrence has only $m$ terms since there is no 
$m$-colored composition of $1$ with parts $\geq 2$.

For initial conditions,
we consider $m$-colored compositions of integers $n$ with $1 \leq n \leq m + 1$.
We have that $d(m, 1) = 0 = \gs{P}{m+1}{m}$ and 
$d(m, 2) = 1 = \gs{P}{m+1}{m + 1}$ for all $m$.
We claim that $d(m, n) = 2^{n-2}$ for $2 \leq n \leq m + 1$.

When $2 \leq n \leq m$,  
the only symbols that can occur in the $m$-colored composition of $n$
are $\{2, 3, 3', 4, 4', 4'', 5, \ldots, n^{(n-2)}\}$
and $d(m, n) = d(m - 1, n)$. 

Consider $d(m, m + 1)$. 
There is exactly one $m$-colored composition of $m + 1$ 
which is not an $(m-1)$-colored composition of $m + 1$, 
namely $x = ((m+1)^{(m-1)})$. 
Hence $d(m, m + 1) = 1 + d(m-1, m+1)$.
Inductively, then, 
\[ d(m, m + 1) = 1 + \sum_{i = 1}^{m} d(m - 1, i)  = 
1 + 0 + \sum_{i = 2}^{m} 2^{i-2} = 2^{m - 1}.\]
\end{proof}

It is again straightforward to see that the sequence $\gs{P}{m}{n}$ 
has the generating function 
\[ p_{m}(x) = \frac{x^{m-2}(1 - x)}{\denom{m}}\]
where
\[ \denom{m} = 1 - x - x^{2} -\ldots - x^{m} = 1 - x \sum_{i = 0}^{m-1} x^{i}\]
as in the generating functions for $\gs{G}{m}{n}$ and $\gs{H}{m}{n}$
from Section~\ref{intro}.

Let $Z = \{2, 3, 3', 4, 4', 4'', \ldots\}$ be the set of colored integers.
Let $\psi$ be a mapping from binary strings $0^{i} 1^{j}$ for $i, j \geq 1$
to $Z$ where $\psi(0^{i} 1^{j}) = (i + j)^{(i-1)}$.

Recall that $\mathcal{L}_{m}$ is the concatenation of 
$\ell_{m} \ell_{m-1} \ldots \ell_{1}$ from \fn, 
where $\ell_{i} = \varepsilon$ for $i > n - 1$.
Let $\Psi(m, n)$ be the $m$-colored multiset obtained by applying 
$\psi$ to the primitives of order $1 \leq i \leq m$ from $\mathcal{L}_{m}$. 
Let $c(m, n, k)$ be the number of integers $k^{(0)} = k \geq 2$ in $\Psi(m, n)$.

\begin{theorem}\label{cmnk}
For $m \geq 1, k \geq 2$,  
the count $c(m, n, k)$
is the coefficient of $x^{n}$ in 
\[ x^{k - m + 1} p_{m+1}(x) = \frac{x^k (1 - x)}{\denom{m+1}}.\]
\end{theorem}

\begin{proof}
The argument is essentially the same as the proof of Theorem~\ref{fibcounts},
except that we consider $m$-colored compositions $x$ of $n - k$ with parts 
greater than $1$ and give the result in terms of generating functions.
The rotational symmetries of $\omega(x)$, where the definition is extended
to higher order primitives, depend on which parts of $x$ are $k^{(0)} = k$. 
\end{proof}
 
By exchanging the colors of $k^{(0)}$ and some $k^{(i)}$ with 
$0 < i \leq \min\{k - 2, m - 1\}$ occurring in the 
$m$-colored compositions of $n - k$, 
we see that 
the number of occurrences of $k^{(i)}$ is also $c(m, n, k)$.

\begin{theorem}
Consider \fn\ and $\mathcal{L}_{m}$ with $n \geq 0$ and $m \geq 1$.
Then $\no{\mathcal{L}_{m}}$ is the coefficient of $x^{n}$ in 
\[ \frac{x^{2} \sum_{i = 0}^{m-1} (i+1) x^{i}}{\denom{m+1}}\]
and $\yes{\mathcal{L}_{m}}$ is the coefficient of $x^{n}$ in 
\[ \frac{x^{2} \sum_{i=0}^{m-1} x^{i}}{(1-x) \denom{m+1}} .\]
\end{theorem}

\begin{proof}
Let $1 \leq i \leq m$ be fixed.
Each primitive of order $i$ contributes $i$ zeros to $\no{\mathcal{L}_{m}}$,
and the number of colored integers $k^{(i-1)}$ in $\Psi(m,n)$ 
is the sum of $c(m, n, k)$ for $i + 1 \leq k \leq n$.
According to Theorem~\ref{cmnk}, this is the sum of the first 
$n - (i + 1) + 1$ terms of the sequence 
$a_{0}, a_{1}, a_{2}, \ldots$ whose generating function is 
\[ \frac{(1-x)}{\denom{m+1}},\] which 
is the coefficient of $x^{n - i - 1}$ in the series
\[ \frac{1}{\denom{m+1}}.\]
The result for $\no{\mathcal{L}_{m}}$ follows by a weighted summation over all 
$1 \leq i \leq m$, with the exponents adjusted appropriately.

Each primitive of order $i$ and length $k$ contributes $(k - i)$ ones to 
$\yes{\mathcal{L}_{m}}$.
To calculate the contribution for a given $i$,  
we again sum over the first $n - i$ terms of the sequence 
$a_{0}, a_{1}, a_{2}, \ldots$, 
except that now each term $a_{j}$ is weighted by $(n - i - j)$.
This is the coefficient of $x^{n - i - 1}$ in the series
\[ \frac{1}{(1 - x) \denom{m+1}}.\]
Summing over the possible $i$'s yields $\yes{\mathcal{L}_{m}}$.
\end{proof}

Recall that $\suff$ is the proper suffix of \fn\ consisting of the 
Lyndon words of length $d \mid n$ containing at most $m$ consecutive $0$'s;
\[ \suff = \mathcal{L}_{m} 1.\] 

\begin{theorem}
Consider \fn\ and \suff\ with $n \geq 1$ and $m \geq 0$.
Then $\sk(\suff)$ is the coefficient of $x^{n}$ in 
\[ \frac{- x + \sum_{i = 3}^{m+1} (i - 2)x^{i}}{\denom{m+1}}\]
and $|\suff|$ is the coefficient of $x^{n}$ in 
\[ \frac{\sum_{i = 1}^{m+1} i x^{i}}{\denom{m+1}}.\]
\end{theorem}

\begin{proof}
We have $\sk(\suff) = \no{\suff} - \yes{\suff}$ where 
$\no{\suff} = \no{\mathcal{L}_{m}}$ 
and $\yes{\suff} = \yes{\mathcal{L}_{m}} + 1$.
Adding $\frac{x}{1-x}$ to the generating function for $\yes{\mathcal{L}_{m}}$
yields 
\[ \frac{x (1 - x^{m+1})}{(1-x) \denom{m+1}}.\]
Taking the difference with $\no{\suff}$ gives 
\[ \frac{-x + \sum_{i = 2}^{m+1} x^{i} - (m-1)x^{m+2}}{(1-x)\denom{m+1}}\]
which simplifies to the desired result.
Similarly, adding the two series yields
\[ \frac{x + \sum_{i = 2}^{m+1} x^{i} - (m+1)x^{m+2}}{(1-x)\denom{m+1}}\]
which again simplifies.
\end{proof}

Offsetting the sequence $\gs{G}{m}{n}$ by an initial zero, 
and recalculating the generating function with the initial $\gs{H}{m}{0} = m$ 
replaced by a zero, we have the result stated in Theorem~\ref{main}.

\bibliographystyle{abbrv}
\bibliography{citepoints}

\end{document}